\documentclass[12pt]{amsart}
\usepackage{amsthm,cleveref,amstext}
\usepackage{enumerate,enumitem}
\usepackage{color,graphicx}
\usepackage[sort,compress,numbers]{natbib}

\usepackage{geometry}
\newtheorem{thm}{Theorem}
\newtheorem{lemma}[thm]{Lemma}
\newtheorem{prop}[thm]{Proposition}
\newtheorem{cor}[thm]{Corollary}

\newtheorem*{remark}{Remark}

\crefname{thm}{Theorem}{Theorems}
\crefname{lemma}{Lemma}{Lemmas}
\crefname{prop}{Proposition}{Propositions}
\crefname{cor}{Corollary}{Corollaries}
\crefname{section}{Section}{Sections}
\crefname{figure}{Figure}{Figures}
\crefname{remark}{Remark}{Remarks}

\newcommand{\df}{\textbf}

\newcommand{\Z}{\mathbb{Z}}
\newcommand{\R}{\mathbb{R}}
\renewcommand{\P}{\mathbb{P}}

\newcommand{\un}{$\mbox{\rm Unif}[0,1]$}

\title[Symmetrization for finitely dependent colouring]{Symmetrization \\ for finitely dependent colouring}
\author{Alexander E.~Holroyd}
\thanks{Funded in part by a Royal Society Wolfson Fellowship}
\address{University of Bristol, U.K.}
\email{a.e.holroyd@bristol.ac.uk}


\date{10 May 2023}

\begin{document}
\maketitle
\begin{abstract}
  We prove the existence of a finitely dependent proper colouring of the integer lattice $\Z^d$ that is fully isometry-invariant in law, for all dimensions~$d$.  Previously this was known only for $d=1$, while only translation-invariant examples were known for higher $d$.  Moreover we show that four colours suffice, and that the colouring can be expressed as an isometry-equivariant finitary factor of an i.i.d.\ process, with exponential tail decay on the coding radius.  Our construction starts from known translation-invariant colourings and applies a symmetrization technique of possible broader utility.
\end{abstract}

\section{Introduction}
\label{intro}

We are concerned with finitely dependent colourings.
A stochastic process $X=(X(v):v\in\Z^d)$ is \df{$k$-dependent} if the families $(X(v):v\in A)$ and $(X(v):v\in B)$ are independent of each other whenever the sets of sites $A,B\subseteq\Z^d$ are at distance greater than $k$ from each other (with respect to the $1$-norm), and $X$ is \df{finitely dependent} if it is $k$-dependent for some finite~$k$.
A process $X$ is a \df{$q$-colouring} if each $X(v)$ takes values in the set of colours $[q]:=\{1,2,\ldots,q\}$, while almost surely $X(u)\neq X(v)$ for any nearest neighbours $u,v\in\Z^d$; it is a \df{colouring} if it is a $q$-colouring for some finite $q$.  A process is \df{translation-invariant} if its law is invariant under all translations of $\Z^d$; \df{isometry-invariance} is defined analogously.

It was recently established in \cite{hl} that there exist translation-invariant finitely dependent colourings of $\Z^d$ for all $d$.  This fact is striking in several ways, even in dimension $d=1$.  In particular, the most obvious way to construct a finitely dependent process is as a \df{block factor} (that is, a finite-range function commuting with translations) of an i.i.d.\ family. (We give full detailed definitions later.)  Regardless of colouring, for some time it was not known whether \emph{every} translation-invariant finitely dependent process can be expressed as a block factor.  This question was considered as early as \cite{ibragimov-linnik}, while the negative answer was proved much later in \cite{burton-goulet-meester}.  See \cite{hl} for further historical information.  On the other hand, Ramsey-theoretic arguments imply that there is no block factor colouring (see \cite{hsw} or \cite{naor}).  Hence the result of \cite{hl} shows that \emph{colouring} distinguishes between block factors and translation-invariant finitely dependent processes.  (In fact, as also shown in \cite{hl}, in $d=1$ the same holds if colouring is replaced with any shift of finite type that is non-trivial in a certain sense.)

At the core of the proof in \cite{hl} is a short but mysterious construction of a finitely dependent colouring 
of the line $\Z$ (i.e.\ $\Z^d$ in dimension $d=1$).  This construction is in a sense the only one of its kind known, with all further developments on the topic employing either generalizations that build on it (e.g.\ \cite{hhl,hl2,levy,star,hhl-cycle}, or embellishments that treat it as a black box (e.g.\ \cite{hl,fin}).  As a particular instance of the latter, given a translation-invariant finitely dependent colouring of the line, it is straightforward to construct one of $\Z^d$ from it, as we explain in \cref{prelim}.  The colouring of the line in \cite{hl} is in fact invariant in law under reflections as well as translations of $\Z$, and thus isometry-invariant.  Surprisingly however, it does not appear straightforward to carry isometry-invariance into higher dimensions, and existence of an isometry-invariant finitely dependent colouring has been open for $d\geq 2$.  It is this question that we address in the main theorem below.

The colouring that we will exhibit has the further property that it admits a well-behaved probabilistic construction in the following sense.  A \df{finitary factor} is a map from one process to another, commuting with translations, in which the output variable at a given location can be determined from the input variables within a finite, but random and perhaps unbounded distance, called the \df{coding radius}.  (Again, detailed definitions appear below).   The factor is \df{isometry-equivariant} if it commutes with all isometries of $\Z^d$.

\begin{thm}\label{main}
  For any $d\geq 2$ there exists an isometry-invariant finitely dependent colouring of $\Z^d$.  Moreover, $4$ colours suffice, and the colouring can be expressed as an isometry-equivariant finitary factor of an i.i.d.\ process, in which the coding radius has exponential tail.
\end{thm}

The analogue of \cref{main} was known already for dimension $d=1$; see \cite{hl} for existence (as mentioned above) and \cite{hhl} for the claim about the finitary factor.  For $d\geq 2$, the number of colours cannot be reduced further.  Indeed, there is not even a translation-invariant finitely dependent $3$-colouring of $\Z^d$ for $d\geq 2$; see \cite{hsw}. Our colouring is $k$-dependent but its dependence distance $k=k(d)$ is very large (even for $d=2$, and even if we allow an arbitrary large number of colours as opposed to just $4$).  We do not know the optimal $k$ for $4$ or any larger number colours, nor indeed whether it equals $1$ for sufficiently many colours.
It is known that even a translation-invariant $1$-dependent colouring of $\Z^d$ requires at least $(d+1)^{d+1}/d^d$ colours, and at least $9$ and $12$ respectively for $d=2$ and $d=3$; see \cite{hl}.

\subsection*{Finitary factors}
We conclude the introduction by formally defining and further discussing finitary factors and related concepts.  The first finitely dependent colouring in \cite{hl} was defined by constructing its finite-dimensional distributions and appealing to Kolmogorov extension, but a direct ``construction on $\Z$'' was missing.  This gap was later filled in \cite{fin,hhl}, with finitary factors providing the natural way to formalize the notion.

Let $X$ and $Y$ be stochastic processes indexed by $\Z^d$ (taking real values, say).  We say that $X$ is a \df{factor} of $Y$ if almost surely $X=f(Y)$ for some measurable function $f$ that is \df{translation-equivariant}, which is to say that $f(\theta y)=\theta f(y)$ for every translation $\theta$ of $\Z^d$ and every configuration $y=(y(v):v\in\Z^d)\in \R^{\Z^d}$.  (Translations act on configurations via $(\theta y)(v):=y(\theta^{-1} v)$ for $v\in\Z^d$.)  \df{Isometry-equivariance} is defined analogously.  Note that (in keeping with standard convention) translation-equivariance is part of the definition of a factor, while isometry-equivariance is an additional stronger property.  When the stronger property is not assumed we sometimes add the descriptor ``translation-equivariant'' for clarity and emphasis.

Let $|\cdot|$ be the $1$-norm on $\Z^d$ and denote the ball $B(r)=\{x\in \Z^d:|r|\leq r\}$.
A factor $f$ from $Y$ to $X$ is \df{finitary} if, for almost every configuration $y$ respect to the law of $Y$, there exists $r<\infty$ such that the image variable $f(y)(0)$ at the origin is determined by the restriction of $y$ to $B(r)$, in the sense that $f(y)(0)=f(y')(0)$ for any $y'$ that agrees with $y$ on $B(r)$.  We let $R=R(y)$ be the smallest such $r$.  The random variable $R=R(Y)$ is then called the \df{coding radius} of the finitary factor.  This is the random distance up to which $Y$ must be examined to determine $X(0)$.  A non-negative random variable $R$ has \df{exponential tail} if $\P(R>r)<C e^{-cr}$ for all $r\geq 0$ and constants $c,C\in(0,\infty)$.  We call a finitary factor whose coding radius has exponential tail \df{exponentially finitary}.  A process $X$ that can be expressed as a finitary factor of some i.i.d.\ process $Y$ is called a finitary factor of i.i.d., or simply \df{ffiid}.  Without loss of generality the i.i.d.\ variables $(Y(v):v\in\Z^d)$ can here be taken to be uniformly distributed on $[0,1]$, denoted \un.  Combining the various concepts, a process may be exponentially ffiid, or isometry-equivariant ffiid, etc. Note that a ffiid process is automatically translation-invariant, and an isometry-equivariant ffiid process is isometry-invariant.

A finitary factor $f$ whose coding radius is a \df{bounded} random variable (so that $R\leq m$ for some finite deterministic $m$) is called a (\df{radius-}$m$) \df{block factor}.  In this case there is a function $g:\R^{B(m)}\to \R$ such that $f(y)(v)=g(y(u): u\in v+B(m))$ for all $v\in\Z^d$ and almost all configurations $u$.  Note that a block factor is in particular exponentially finitary.

Finitary factors and their interactions with colouring and other shifts of finite type are intimately tied to distributed algorithms and descriptive combinatorics.  See \cite{local} for a detailed discussion.  In particular, an exponentially ffiid process corresponds to a local algorithm that requires $O(\log n)$ steps on a portion of the graph of diameter $n$, while automorphism-equivariance corresponds to a local algorithm that works in an undirected setting.

\section{Preliminaries and proof outline}
\label{prelim}

As remarked above, existence of translation-invariant finitely dependent colourings was first proved in \cite{hl}.  The key case is dimension $d=1$, where the colourings are in fact isometry-equivariant.

\begin{thm}[\cite{hl}]\label{1d}
  There exist an isometry-invariant $1$-dependent $4$-colouring and $2$-dependent $3$-colouring of $\Z$.
\end{thm}

On the other hand there is no translation-invariant $1$-dependent $3$-colouring of $\Z$; see \cite{hsw} or \cite{hl}.  Turning to higher dimensions, the following is a straightforward consequence of \cref{1d}, but note crucially that isometry-invariance is weakened to translation-invariance.

\begin{cor}[\cite{hl}]\label{trinv}
  For any $d\geq 2$ there exists a translation-invariant $1$-dependent $4^d$-colouring of $\Z^d$.
\end{cor}

To illustrate the difficulty to be overcome in strengthening translation-invariance to isometry-invariance, we explain the simple construction behind the above corollary.  For simplicity we restrict attention to the key case $d=2$.  For each $j\in\Z$ we let $Y^j=(Y^j(i):i\in \Z)$ and $Z^j=(Z^j(i):i\in \Z)$ be copies of the $1$-dependent $4$-colouring from \cref{1d}, where the processes $(Y^j,Z^j:j\in\Z)$ are all independent of each other.  Now let
$$X(i,j)=\bigl(Y^i(j),Z^j(i)\bigr),\quad (i,j)\in\Z^2.$$
In other words, we take a colouring of each horizonal and each vertical line, and combine the two colours of the two lines that intersect at a site into an ordered pair placed at that site.
It is easy to check that $X=(X(i,j):(i,j)\in\Z^2)$ is a translation-invariant $1$-dependent $16$-colouring of $\Z^2$, where the set of colours is the set of ordered pairs $[4]^2$ (which is obviously in bijection with $[4^2]$).

Since the colouring on $\Z$ was reflection-invariant, $X$ is also invariant in law under the reflection in each coordinate axis.  However, it is not invariant under the reflection $(i,j)\mapsto(j,i)$ that exchanges the two coordinates.  The difficulty is that we had to make a choice about which way to order the ordered pair $X(i,j)=(\cdot,\cdot)$, and exchanging the coordinates disrupts it.  This obstruction might appear to be a trivial technical one, but we know of no simple way around it.  The main contribution of this article is a not-so-simple way around.


One may hope for a process of interest to be ffiid, preferably with good tail decay of the coding radius.  It was shown in \cite{fin} that the $1$-dependent $4$-colouring in \cref{1d} is ffiid (via an explicit factor), while in \cite{spinka} it was shown (via a general argument) that every translation-invariant finitely dependent process on $\Z$ is ffiid.  However, in both cases the coding radius of the constructed factor has infinite mean.  This was improved on by the following result, which was obtained by via a surprising generalization of the construction of \cite{hl} from uniform to Mallows-distributed permutations.

\begin{thm}[\cite{hhl}]\label{hhl}
  For each of $(k,q)=(1,5),(2,4),(3,3)$ there exists a $k$-dependent $q$-colouring of $\Z$  that is isometry-equivariant exponentially ffiid.
\end{thm}

Recall that no colouring exists for $(k,q)=(1,3)$, while the two original colourings of \cite{hl} had $(k,q)=(1,4),(2,3)$.  In these last two cases it is not known whether there exist colourings that are exponentially ffiid, or ffiid with finite mean coding radius, nor indeed whether either property holds for the known colourings from \cite{hl}.

We will treat \cref{hhl} as a black box, and build higher-dimensional colourings from the one-dimensional ones.  (In fact we need only one case, say $(k,q)=(1,5)$, and we need only translation-equivariance on $\Z$).  In particular, we will make extensive use of block factors, and we rely on the following simple facts.

\begin{samepage}
\begin{lemma}\ \label{composition}
\begin{enumerate}[label={\rm (\roman*)}]
  \item A composition of translation-equivariant maps is translation-equivariant, and similarly for isometry-equivariance.
  \item A composition of exponentially finitary factors is an exponentially finitary factor.
  \item A radius-$r$ block factor of a $k$-dependent process is $(k+2r)$-dependent.
\end{enumerate}
\end{lemma}
\end{samepage}

\begin{proof}
Parts (i) and (iii) follow immediately from the definitions.  Turning to (ii), let $X=f(Y)$ and $Y=g(Z)$ where $f$ and $g$ are exponentially finitary factors with coding radii $R$ and $S$ respectively.  Also for $v\in\Z^d$ write $S_v=S(\theta^{-v} Z)$, where $\theta^{-v}$ is the translation of $\Z^d$ by $-v$, for the coding radius of $g$ at location $v$, i.e.\ the radius around $v$ to which we need to examine $Z$ to determine $Y(v)$.  Then $X(0)$ can be determined from the restriction of $Z$ to $B(T)$ where
$$T=\max\bigl\{|v|+S_v:v\in B(R)\bigr\}.$$  But we have for any $t\geq 0$,
\begin{align*}
  \P(T>2t)&\leq \P(R>t)+\P(R\leq t,\ S_v>t \text{ for some }v\in B(t))\\
  &\leq C_1 e^{-c_1 t}+C_2 t^d C_3 e^{-c_3 t}\\
  &\leq C_4 e^{-c_4 t},
\end{align*}
for constants $C_i,c_i\in(0,\infty)$.
\end{proof}

Note in particular that by \cref{composition}(ii), a block factor of an exponentially ffiid process is also exponentially ffiid.

Our proof will proceed via a sequence of intermediate processes, whose characteristics we describe next.  Let $m$ be a positive integer and let $\Z^d[m]$ denote the graph with vertex set $\Z^d$ and with an edge between distinct vertices $u,v$ whenever $|u-v|\leq m$.  A \df{range-$m$ $q$-colouring} is a process $X=(X(v): v\in\Z^d)$ in which each $X(v)$ takes values in $[q]$ such that a.s.\ $X(u)\neq X(v)$ whenever $u$ and $v$ are adjacent in $\Z^d[m]$.

A range-$m$ colour-$a$ \df{tile} of a process $X$ is the vertex set of a connected component of the subgraph of $\Z^d[m]$ induced by the set of vertices $v$ with $X(v)=a$.  A process $X$ is a range-$m$ \df{$b$-bounded $q$-tiling} if it is $[q]$-valued and each of its range-$m$ tiles has at most $b$ elements.  As before, we omit parameters to indicate that they are allowed to take an unspecified finite value, so for example a range-$m$ bounded tiling means a process that is a range-$m$ $b$-bounded $q$-tiling for some finite $b$ and $q$.  Note that a range-$1$ $1$-bounded $q$-tiling is simply a $q$-colouring.

We say that a $\{0,1\}$-valued process $J=(J(v):v\in\Z^d)$ is an $(a,b)$-\df{net} if a.s.\ there do not exist distinct $u,v\in\Z^d$ with $\|u-v\|\leq a$ and $J(u)=J(v)=1$, but for every $u\in\Z^d$ there exists $v$ with $\|u-v\|\leq b$ and $J(v)=1$.  Note that the support of an $(m,m)$-net is precisely a maximal independent subset of $\Z^d[m]$.

\subsection*{Steps of the construction} The plan is as follows.  Fix $d\geq 1$ and let $m\geq 1$ be arbitrary.  For each of (i)--(ii) below we will show that there exists a finitely dependent process on $\Z^d$ that is (\emph{translation}-equivariant) exponentially ffiid with the stated property:
\begin{enumerate}[label=(\roman*)]
\item a range-$m$ colouring;
\item an $(m,m)$-net.
\end{enumerate}

Then for each of (iii)--(vii) below we will show that there exists a finitely dependent process that is \emph{isometry}-equivariant exponentially ffiid with the stated property:
\begin{enumerate}[resume,label=(\roman*)]
\item the maximum of $c$ $(m,m)$-nets, for some $c=c(d)$;
\item an $(m,c m)$-net, for some $c=c(d)$;
\item a range-$m$ bounded tiling;
\item a range-$m$ colouring;
\item a $4$-colouring.
\end{enumerate}
The maximum referred to in (iii) is site-wise, so the support of the process is the union of the supports of the $c$ nets.

For each of (ii)--(vii) we will deduce the existence of such a process from the case immediately before it in the list (sometimes with a different $m$).  The first step (i) comes from the $d=1$ case, \cref{hhl}.  The key symmetrization step is from (ii) to (iii), where we pass from translation- to isometry-equivariance.  If we only require an isometry-\emph{invariant} colouring rather than an isometry-equivariant factor as our final process, then this step becomes very simple: we simply take the maximum of a family of nets that are images of i.i.d.\ copies of a translation-invariant net under each of the finitely many isometries of $\Z^d$ that fix the origin.  Obtaining a factor requires a more subtle method in which we first split and distribute the underlying i.i.d.\ variables appropriately.  This is addressed in \cref{init}.   Each of the other steps involves applying a suitably designed block factor.  The last step from (vi) to (vii) uses a construction from \cite{hsw}, for which $m$ needs to be at least some large constant depending on $d$.  If we require only a colouring rather than a $4$-colouring, then this step can be omitted, and $m$ can be taken to be $1$ in (v) and (vi), simplifying the proof somewhat.

In several of the steps we will need to add extra randomness, for which the following definition will be convenient.  Let $Y=(Y(v):v\in\Z^d)$ be a process and let $U=(U(v):v\in\Z^d)$ be a family of i.i.d.\ variables (uniformly distributed on $[0,1]$, say), independent of $Y$.  We denote the combined process $(Y,U)=((Y(v),U(v)):v\in\Z^d)$.  Suppose that $X$ is a block factor of $(Y,U)$.  Then we say that $X$ is a \df{randomized block factor} of $Y$.  Note that if $Y$ is respectively (exponentially) ffiid or finitely dependent then so is $(Y,U)$, so randomized block factors behave similarly to their deterministic counterparts for our purposes.

\section{Initial steps}
\label{init}

Steps (i) and (ii) of the plan described in the last section are relatively straightforward, as follows.

\begin{prop}\label{trans-col}
  Fix integers $d\geq 1$ and $m\geq 1$.  There exists a finitely dependent range-$m$ colouring of $\Z^d$ that is (translation-equivariant) exponentially ffiid.
\end{prop}

\begin{proof}
  We apply the construction of \cite[Corollary 20]{hl} to the $1$-dimensional ffiid colouring of \cref{hhl}, adapted so as to yield a factor. The approach is an extension of the proof of \cref{trinv} discussed earlier.  A \df{line} is an infinite set of the form $\{a+ih:i\in\Z\}\subset\Z^d$ where $a,h\in\Z^d$, and $h\neq 0$ is called its \df{direction}.  Let $H$ be a fixed set containing exactly one of $h$ and $-h$ for each $h\in B(m)\setminus\{0\}$.  The idea is to assign a $1$-dependent colouring to each line with direction in $H$, and combine the colours of all lines passing through a site.  We need to do this via a factor.  By \cref{hhl}, let $f$ be an exponentially finitary factor that maps a process of i.i.d.\ \un\ variables $U=(U(i):i\in\Z)$ to a $1$-dependent $5$-colouring of $\Z$.  Let $(U(v,h): v\in \Z^d,h\in H)$ be i.i.d.\ \un\ variables indexed by sites and directions.  Also write $W(v)=(U(v,h):h\in H)$.

  For any line $L=\{a+ih:i\in\Z\}$ with direction $h\in H$, let $U^L=(U(a+ih,h):i\in\Z)$ and define a colouring $Y^L$ of $L$ by $Y^L(a+ih) = f(U^L)(i)$ for $i\in\Z$.  For a site $v$ and a direction $h\in H$, let $L(v,h)$ be the unique line of direction $h$ passing through $v$.  Now assign to any site $v$ the tuple of colours
  $$X(v):=(Y^{L(v,h)} (v):h\in H)\in[5]^H.$$
  It is straightforward to check that $X$ is $m$-dependent and a range-$m$ colouring (see \cite{hl} for details), and an exponentially finitary factor of the i.i.d.\ process $W$.
\end{proof}

\begin{lemma}\label{net-factor}
   Fix integers $d,q,m\geq 1$. There exists an isometry-equivariant block factor $f$ such that if $X$ is a range-$m$ $q$-colouring of $\Z^d$ then $f(X)$ is an $(m,m)$-net of $\Z^d$.
\end{lemma}

\begin{proof}
  This is standard; see e.g.\ \cite[Corollary 11]{hsw}.  For the reader's convenience we briefly describe the construction.  Recall that we want a process whose support is a maximal independent set of $\Z^d[m]$.  For each $2\leq k\leq q$, we define the map $\mathcal{E}_k: [q]^{\Z^d}\to [q]^{\Z^d}$ between configurations that replaces colour $k$ with $1$ wherever it has no neighbouring $1$s in $\Z^d[m]$:
  $$(\mathcal{E}_k x)(v)=
  \begin{cases}
  1,&x(v)=k \text{ and }1\notin\{x(u):u\in v+B(m)\} \\
  x(v),&\text{otherwise.}
  \end{cases}
  $$
  Note that this is a radius-$1$ isometry-equivariant block factor.
  Now to define $f$ we apply each of these maps in sequence and finally keep only the $1$s:
  \belowdisplayskip=-12pt
  \[(f(x))(v) =
  \begin{cases}
  1&(\mathcal{E}_2 \mathcal{E}_3 \cdots \mathcal{E}_q x)(v)=1 \\
  0&\text{otherwise.}
  \end{cases}
  \]
\end{proof}

\begin{cor}\label{net}
  Fix integers $d,m\geq 1$. There exists a finitely dependent exponentially ffiid $(m,m)$-net of $\Z^d$.
\end{cor}

\begin{proof}
  Apply the block factor of \cref{net-factor} to the colouring of \cref{trans-col}, and use \cref{composition}.
\end{proof}

\section{Symmetrization}
\label{sym}

In this section we perform the step from (ii) to (iii) as described in \cref{prelim}.  As mentioned earlier, this would be easy if we merely required an isometry-invariant process, but obtaining an isometry-equivariant factor requires a further preparatory step involving the i.i.d.\ variables.

Let $\Gamma$ be the group of all isometries of $\Z^d$ that preserve the origin, and let $$\kappa=\kappa(d):=\#\Gamma=2^d d!$$ be its cardinality.  (This expression arises because the $d$ coordinates can be arbitrarily permuted, and each can be negated).
The following device will be useful.  Define the fixed vector $$\rho:=(1,2,\ldots,d)\in\Z^d,$$ and note that its images under elements of $\Gamma$ are all distinct.  (This is the only property of $\rho$ that we care about).

We will introduce a family of i.i.d.\ \un\ random variables indexed by pairs of sites whose difference is such an image of $\rho$.  We denote the set of such pairs by
$$\Xi_d:=\bigl\{(x,y)\in (\Z^d)^2: y-x\in\Gamma\rho\bigr\},$$
and we will introduce an i.i.d.\ family $W=(W(x,y): (x,y)\in\Xi_d)$.  We think of the variable $W(x,y)$ as located at the site $x$ but associated with an arrow pointing in the direction from $x$ to $y$.  These variables will be constructed from an i.i.d.\ family $U=(U(x):x\in\Z^d)$ in an isometry-equivariant way.  To make the latter notion precise, recall that an isometry $\theta$ of $\Z^d$ acts on a configuration $u=(u(x): x\in\Z^d)$ via $(\theta u)(x):=u(\theta^{-1} x)$; similary for a configuration $w=(w(x,y): (x,y)\in\Xi_d)$ we take $(\theta w)(x,y):=w(\theta^{-1}x,\theta^{-1}y)$.

\begin{prop}\label{dist}
  Let $d\geq 1$.  There exists a measurable map $F:[0,1]^{\Z^d}\to [0,1]^{\Xi_d}$ such that:
  \begin{enumerate}[label=(\roman*)]
  \item if $W=F(U)$, where $U=(U(x):x\in\Z^d)$ are i.i.d.\ \un, then $W=(W(x,y):(x,y)\in\Xi_d)$ are i.i.d.\ \un;
  \item $F$ is isometry-equivariant in the sense that $F(\theta u)=\theta F(u)$ for every isometry $\theta$ of $\Z^d$;
  \item $F$ is a $d$-block-factor in the sense that $(F(u)(0,w):w\in\Gamma\rho)$ is a function only of $(u(v):v\in B(d))$.
  \end{enumerate}
\end{prop}

\begin{proof}
We describe the map by constructing $W=F(U)$ for $U$ i.i.d.\ \un.  Let $s$ be a fixed function that maps a \un\ random variable to a $(\kappa+1)$-vector of i.i.d.\ \un\ random variables (for example, by taking disjoint subsequences of  the binary expansion).  For each $x\in\Z^d$, let $(Y(x),Z_1(x),\ldots,Z_\kappa(x))=s(U(x))$.  We now use the $Y$ variables to reorder the $Z$ variables.  Specifically, for $x,y$ with $y-x\in\Gamma\rho$, let
\begin{gather*}
W(x,y)=Z_T(x), \text{ where} \\
Y(y) \text{ is the $T$th largest among } \bigl\{Y(z): z-x\in\Gamma\rho\bigr\}.
\end{gather*}
(On the null event where two of more elements of the latter set are equal, we can for example take $W(x,y)=0$ for all $y$.) The variables $(W(x,y):(x,y)\in\Xi_d)$ are all i.i.d.\ \un, since this holds even conditional on $((Y(x):x\in\Z^d)$.  The remaining properties (ii) and (iii) are clear from the construction.
\end{proof}

\begin{remark}
  In the above proof and others to follow, the map $F$ is described via its action on the \emph{random} configuration $U$ (always denoted by an upper case letter).  Strictly speaking we need to define $F(u)$ for every \emph{deterministic} configuration $u$ in an appropriate domain.
  However, such a definition is implicit in our description.  Making it explicit would involve a burdensome proliferation of symbols while adding nothing of significance.
\end{remark}

Now we are ready to apply the above tool.

\begin{prop}\label{max}
  For any $d\geq 1$ and $m\geq 1$ there exists a process on $\Z^d$ that is the maximum of $\kappa$ $(m,m)$-nets, and which is finitely dependent and isometry-equivariant exponentially ffiid.
\end{prop}

\begin{proof}
  Let $U=(U(x):x\in\Z^d)$ be i.i.d.\ \un.  We will construct the required process as a factor of $U$.  First let $F$ be the map from \cref{dist} and let $W=F(U)$, so that $(W(x,y):(x,y)\in\Xi_d)$ are i.i.d.\ \un.  For each $\gamma\in\Gamma$, define $W_\gamma=(W_\gamma(x):x\in\Z^d)$ by $W_\gamma(x):=W(x,x+\gamma\rho)$.  Let $G$ be the translation-equivariant map from \cref{net} that maps an i.i.d.\ \un\ family to an $(m,m)$-net.  Now define $J_\gamma=(J_\gamma(x):x\in\Z^d)$ by
  $$J_\gamma = \gamma G(\gamma^{-1} W_\gamma).$$
  Finally, set $$J=\max_{\gamma\in\Gamma} J_\gamma.$$  Write $H$ for the map just constructed, so that $J=H(U)$.

  Clearly $J$ is a maximum of $\kappa$ $(m,m)$-nets.  Each process $J_\gamma$ is finitely dependent (by \cref{composition}), and they are independent of each other for different $\gamma$, so $J$ is finitely dependent.

  We need to show that the map $H$ is isometry-equivariant.  This is sufficiently delicate that we check it explicitly.  Let $\theta$ be an isometry of $\Z^d$, and define $U'=\theta U$.  We denote the corresponding objects $W'=F(U')$; $W'_\gamma(x)=W'(x,x+\gamma\rho)$; $J'_\gamma = \gamma G(\gamma^{-1} W'_\gamma)$ and $J'=\max_\gamma J'_\gamma$, so that $J'=H(U')$.  We need to show that $J'=\theta J$, and it is enough to do this for $\theta\in\Gamma$ and for $\theta$ a translation, since any isometry can be expressed as a composition of these.
  First note that since $F$ is isometry-equivariant, we have $W'=F(\theta U)=\theta W$.  Therefore, for any $x\in\Z^d$ and $\gamma\in\Gamma$ we have \begin{equation}W'_\gamma(x)=(\theta W)(x,x+\gamma\rho)=W(\theta^{-1}x,\theta^{-1}(x+\gamma\rho)).\label{step}\end{equation}

  Supposing that $\theta\in\Gamma$ we deduce that
  $$(\gamma^{-1}W'_\gamma)(x)=W_\gamma'(\gamma x)=W(\theta^{-1}\gamma x,\theta^{-1}\gamma x+\theta^{-1}\gamma\rho)=W_{\theta^{-1}\gamma}(\theta^{-1}\gamma x),$$
  which is to say
  $$\gamma^{-1}W'_\gamma = \gamma^{-1}\theta W_{\theta^{-1}\gamma}.$$
  Now it follows that
  $$J'_\gamma = \gamma G(\gamma^{-1}\theta W_{\theta^{-1}\gamma})=\theta \theta^{-1}\gamma G(\gamma^{-1}\theta W_{\theta^{-1}\gamma}) = \theta J_{\theta^{-1}\gamma}.$$
  Then, since $\Gamma$ is a group,
  $$J'=\max_{\gamma\in \Gamma} J'_\gamma=\theta \max_{\gamma\in \Gamma} J_{\theta^{-1}\gamma} =  \theta \max_{\gamma'\in \Gamma} J_{\gamma'}=\theta J.$$

  On the other hand, if $\theta$ is a translation then \eqref{step} implies that $W'_\gamma=\theta W_\gamma$.  Since the conjugate $\gamma^{-1}\theta\gamma$ is also a translation, and $G$ is translation-equivariant, we have
  $$J'_\gamma=\gamma G(\gamma^{-1}\theta W_\gamma)=\gamma G(\gamma^{-1}\theta\gamma\gamma^{-1} W_\gamma)=\gamma \gamma^{-1}\theta\gamma G(\gamma^{-1} W_\gamma)=\theta J_\gamma,$$
  and hence $J'=\theta J$ in this case also, as required.
\end{proof}

\section{Concluding steps}
\label{conc}

We now perform the remaining steps from (iii) to (vii) as outlined in \cref{prelim}.

\begin{prop}\label{clusters} Fix integers $d,m,k\geq 1$.  There exists a randomized isometry-equivariant block-factor on $\Z^d$ that maps any maximum of $k$ $(k m,k m)$-nets to an $(m,2k m)$-net.
\end{prop}

\begin{proof}
  Let $J_1,\ldots,J_k$ be $(k m,k m)$-nets and let $J=\max_i J_i$ be their maximum.  It is also convenient to denote their respective supports by $S_i:=\{v:J_i(v)=1\}$ and $S:=\{v: J(v)=1\}=\bigcup_i S_i$.

  Define a \df{cluster} to be the vertex set of a connected component of the subgraph of $\Z^d[m]$ induced by $S$.  We claim that each cluster contains at most one site from each of the sets $S_1,\ldots,S_k$.  Suppose this fails, and let $x_0,x_1,\ldots ,x_r$ be the sites of a \emph{minimal} path in $\Z^d[m]$ connecting two sites that lie in the same set.  So without loss of generality, $x_0,x_r\in S_1$ while $x_1,\ldots,x_{r-1}\in\bigcup_{i=2}^{k} S_i$ with no two of $x_1,\ldots,x_{r-1}$ belonging to the same $S_i$.  But then $r\leq k$ and so $|x_0-x_r|\leq k m$, which is a contradiction since $J_1$ is a $(k m,k m)$-net, thus proving the claim.  In particular, each cluster has diameter at most $(k-1)m$.

%

   To construct the required net $I$ we now choose a random element from each cluster.  More explicitly, let $(U(i):i\in\Z^d)$ be i.i.d.\ $\mbox{Unif}[0,1]$ and independent of $J$.  For each cluster $A$, let $w=w(A)\in A$ be such that $U(w)=\max\{U(u):u\in A\}$.  For $u\in \Z^d$, let $I(v)=1$ if $v=w(A)$ for some cluster $A$, otherwise let $I(v)=0$. Write $R:=\{v: I(v)=1\}$ for the support of $I$.

  Since $R$ contains only one site from each cluster, no two sites of $R$ are within distance $m$ of each other.  On the other hand, for any site $u\in\Z^d$ there is a site $v$ of $S_1$ within distance $k m$, and there is a site $w$ of $R$ in the same cluster as $v$, so $|u-w|\leq k m+(k-1)m\leq 2k m$.  Therefore $I$ is an $(m,2k m)$-net.

  Note that $I$ can be determined from $(J,U)$ without examining the individual processes $J_1,\ldots,J_k$ (even though our arguments involve them).  Moreover, $I(v)$ can be determined by examining the cluster of $v$ and its boundary in $\Z^d[m]$, which extends to at most distance $(k-1)m+m\leq k m$.  Therefore $I$ is a $k m$-block-factor of $(J,U)$.  Isometry-equivariance is clear from the construction.
\end{proof}

\begin{prop}\label{tiling} Fix integers $d,m,b\geq 1$, and let $\kappa=\#\Gamma$ as usual.  There exists an isometry-equivariant block-factor on $\Z^d$ that maps any $(\kappa m ,b)$-net to a range-$m$ $\kappa$-bounded tiling.
\end{prop}

\begin{proof}
  We first define a fixed deterministic integer labelling $a=(a(v):v\in\Z^d)$ of $\Z^d$ as follows.  Let $A_0,A_1,A_2,\dots$  be the orbits of $\Z^d$ under the action of $\Gamma$, indexed in any non-decreasing order of the $1$-norm of an element of the orbit, so that for $u\in A_i$ and $v\in A_j$ with $i\leq j$ we have $|u|\leq|v|$.  (The norm is constant on each orbit, and we choose a fixed arbitrary ordering among orbits having equal norm).  Thus, for example, $A_0=\{0\}$, while $A_1$ consists of the $d$ standard basic vectors and their negatives, and $A_2$ and $A_3$ are equal (in either order) to $2A_1$, and the set of vectors with two nonzero coordinates each equal to $\pm1$.  Now for each $k=0,1,\ldots$ we set
$a(v) = k$ for all $v\in A_k$.

Suppose $J$ is a $(\kappa m ,b)$-net, and define $Y$ by
$$Y(v) = \min\{a(v-u): J(u)=1\}.$$
That is, each $1$ of $J$ tries to paint $\Z^d$ with a copy of $a$ centered at itself, and the smallest label (which comes from a closest $1$) wins.  This will be the required tiling.  It is clear that the map from $J$ to $Y$ is  isometry-equivariant.  Since $J$ is a $(\kappa m ,b)$-net, for every site $v$ we have $Y(v)\leq s$ where $s=\max\{a(u):|u|=b\}$.  By the same reasoning, the map is a radius-$b$ block factor.

It remains only to check that no range-$m$ tile of $Y$ has more that $\kappa$ sites.  To prove this, suppose on the contrary that distinct sites $x_1,\ldots,x_{\kappa+1}$ induce a connected subgraph of $\Z^d[m]$, where $Y(x_i)=r$ (say) for all $i$.  By the definition of $Y$ there exist $u_1,\ldots,u_{\kappa+1}$ in the support of $J$ with $x_i-u_i\in A_r$ for each $i$.  Since $\# A_r\leq \# \Gamma=\kappa$, by the pigeonhole principle there exist distinct $x_i,x_j$ with $x_i-u_i=x_j-u_j$.  But then $|u_i-u_j|=|x_i-x_j|\leq \kappa m$, contradicting the fact that $J$ is a $(\kappa m ,b)$-net.
\end{proof}

\begin{prop}\label{patchwork} Fix integers $d,m,c,q\geq 1$.  There exists a randomized isometry-equivariant block-factor on $\Z^d$ that maps any range-$m$ $c$-bounded $q$-tiling to a range-$m$ $cq$-colouring.
\end{prop}

\begin{proof}
The idea is to assign distinct colours to each site of a tile at random.
Let $Y$ be the tiling, and let $(U(v):v\in\Z^d)$ be i.i.d.\ $\mbox{Unif}[0,1]$ and independent of $Y$.  For a site $v$, set $X(v)=(j,k)$ if $Y(v)=j$ and $U(v)$ is the $k$th largest of the variables $U(u)$ for $u$ in the same tile as $v$.  The process $X$ is clearly a colouring as required.  Since each tile of $Y$ has diameter at most $cm$, and a tile can be determined by examining it and its outer boundary in $\Z^d[m]$, the map is a radius-$(c+1)m$ block-factor.
\end{proof}

\begin{prop}[Corollary 15 of \cite{hsw}]\label{hsw} Fix $d\geq 1$.  There exists $m=m(d)<\infty$ such that for any $q$ there exists a randomized isometry-equivariant block-factor on $\Z^d$ that maps any range-$m$ $q$-colouring to a (range-$1$) $4$-colouring.
\end{prop}

\begin{proof}[Proof of \cref{main}]
Let $m=m(d)$ be as in \cref{hsw}.  Consider a maximum of $\kappa$ $(\kappa m,\kappa m)$-nets that is finitely dependent and isometry-equivariant exponentially ffiid as provided by \cref{max}.  Then apply \cref{clusters,tiling,patchwork,hsw} in sequence to obtain a $(\kappa m,2\kappa^2 m)$-net, a range-$m$ bounded tiling, a range-$m$ colouring, and finally a $4$-colouring, which has the required properties by \cref{composition}.
\end{proof}

\subsection*{Bounding the dependence distance}
We briefly explain how to obtain an explicit bound on the dependence distance in the colouring of \cref{main}.  As mentioned earlier, the bound is very large. It would be of interest to improve it to something more reasonable.

From the proofs in \cite{hsw} we have following quantitative version of \cref{hsw}.  We can take $m=70d\cdot 10^d$; then given a range-$m$ $q$-colouring, there is a radius-$mq$ block factor mapping it to a $4$-colouring.  (To extract this information from the proof of Theorem 1(i) of \cite{hsw}, in the notation of that paper we can take $C\leq 5^d/(1/2)^d$ by the proof of Lemma 12 of \cite{hsw}, and then $M=14dC+1$ and $m\geq 4M+3$.  Then the construction of the $4$-colouring has two stages, each involving applying $q$ successive radius-$m$ block factors).

\cref{trans-col} then gives a $\kappa m$-dependent range-$\kappa m$ $q$-colouring where $q=5^v$ and $v\leq \# B(\kappa m)\leq (3\kappa m)^d\leq (200d)^{d^2}\leq d^{6d^2}$, and \cref{net-factor} gives a $(\kappa m,\kappa m)$-net via a radius-$q\kappa m$ block factor, adding $2q\kappa m$ to the dependence distance by \cref{composition}(iii).  The subsequent steps involve block factors of radii $d$, $\kappa^2m$, $2\kappa^2 m$, $(3\kappa^2 m)^d$, $(\kappa+1)m$, and $m\kappa (3\kappa^2 m)^d$, all of which are relatively insignificant compared with $q$, and so can be accounted for by increasing the bottom constant in the bound $q\leq 5^{d^{6 d^2}}$.  Therefore the final $4$-colouring of \cref{main} is $k$-dependent where $k=6^{d^{d^2}}$.

The bound on $k$ can be reduced somewhat if we require only a colouring (rather than a $4$-colouring).  However, it is still at least exponential in $d!$, and very large even when $d=2$. 

\section*{Open Questions}

\begin{enumerate}
  \item Does there exist an isometry-invariant $1$-dependent colouring of $\Z^2$ (or indeed of $\Z^d$ for any $d\geq 2$)?  Recall that the answer becomes yes on $\Z$, or if we weaken the requirement to translation-invariance, or (as proved here) if we weaken $1$-dependence to $k$-dependence for large enough $k$.
  \item Does there exist an automorphism-invariant finitely dependent colouring of a $3$-regular tree (or a $d$-regular tree for any $d\geq 3$)?  Similarly to $\Z^d$, it is not difficult to obtain a finitely dependent colouring that is invariant under \emph{some} transitive group of automorphisms (see \cite{hl}), but the symmetrization method of this article no longer works because the stabilizer of a vertex is now infinite (in fact, uncountable).  See \cite{star} for some potential first steps on this question.
  \item Does there exist an ffiid finitely dependent colouring of $\Z$ (or indeed of $\Z^d$) with strictly faster than exponential tail decay of the coding radius?  Since a block factor is impossible this might seem somewhat unlikely, but note that without the requirement of finite dependence, the best decay for an ffiid colouring is a tower function -- i.e., $\P(R>r)<1/\exp\cdots \exp 1$ where the exponential is iterated $c r$ times (see \cite{hsw}).
  \item Does there exist a translation-invariant simple point process on $\R$ that is finitely dependent in the sense that its restrictions to $(-\infty,0]$ and $[c,\infty)$ are independent for some $c$, and such that almost surely each gap between two consecutive points lies in a fixed interval $[a,b]$?  This question was suggested by Ken Alexander (personal communication).  Such a process is the continuum analogue of a finitely dependent net, but the obvious discretization approach seems to fail.
      Similar questions are possible in higher dimensions.
\end{enumerate}

\newpage
\section*{Acknowledgments}

I thank Sebastian Brandt, Tom Hutchcroft and Avi Levy for valuable discussions. In particular, Sebastian Brandt suggested the idea behind the proof of \cref{tiling}.  Key parts of this work were conducted at an excellent workshop of the Oberwolfach Research Institute organised by Jan Greb\'ik, Oleg Pikhurko and Anush Tserunyan.  I thank the institute and the organisers.

\bibliographystyle{abbrv}
\bibliography{isom}

\end{document}